\documentclass[english,12pt,a4paper]{amsart} 

\usepackage{babel}
\usepackage{amssymb}
\usepackage{amsmath}
\usepackage{stmaryrd}
\usepackage{amsthm}
\usepackage{amscd}
\usepackage{pifont}
\usepackage{textcomp}
\usepackage[mathscr]{euscript}
\usepackage{hyperref}
\usepackage[top=4cm, bottom=3cm, left=3.5cm, right=3.5cm]{geometry}

\usepackage[small,nohug]{diagrams}
\diagramstyle[labelstyle=\scriptstyle]

\theoremstyle{plain}
\newtheorem{Theorem}{Theorem}[section]
\newtheorem{Proposition}[Theorem]{Proposition}
\newtheorem{Lemma}[Theorem]{Lemma}

\newtheorem{Corollary}[Theorem]{Corollary}

\newtheorem{Remark}[Theorem]{Remark}
\newtheorem{Conjecture}[Theorem]{Conjecture}

\newcommand{\ol}{\bar}
\newcommand{\bfsigma}{\mbox{\boldmath$\sigma$}}
\newcommand{\Gal}{{\rm Gal}}

\def\hefresh{
   \def\hefreshD{\mathop{\raise1.5pt\hbox{${\smallsetminus}$}}}
   \def\hefreshS{\mathop{\raise0.85pt\hbox{$\scriptstyle\smallsetminus$}}}
   \mathchoice{\hefreshD}{\hefreshD}{\hefreshS}{\hefreshS}}
\renewcommand{\setminus}{\hefresh}

\def\Pp{{\mathbb{P}}}

\def\Aa{{\mathbb{A}}}
\def\Nn{{\mathbb{N}}}
\newcommand{\Spec}{\mathrm{Spec}}

\newcommand{\mf}{\mathfrak}

\begin{document}

\bibliographystyle{alpha}

\title[Ranks of abelian varieties and Mordell-Lang]{Ranks of abelian varieties and the full Mordell-Lang conjecture in dimension one}
\author{Arno Fehm and Sebastian Petersen}

\maketitle

\begin{abstract}
Let $A$ be a non-zero abelian variety
over a field $F$ that is not algebraic over a finite field. 
We prove that 
the rational rank
of the abelian group $A(F)$ is infinite
when $F$ is large in the sense of Pop (also called ample).
The main ingredient is a deduction of the 1-dimensional case of the
relative Mordell-Lang conjecture from a result of R\"ossler.
\end{abstract}

\par\medskip

\section{Introduction}

\noindent
A considerable amount of work has been done to show that
over certain fields,
for example over certain infinite algebraic extensions of global fields,
all abelian varieties have infinite rank.
Here, the {\bf rank} 
of an abelian variety $A$ over a field $F$ is
the rank of the Mordell-Weil group $A(F)$, i.e.~${\rm dim}_{\mathbb{Q}}(A(F)\otimes\mathbb{Q})$.
See \cite{Rosen1973,FreyJarden1974,RosenWong2002,Larsen2003,Im2006, geyerjarden2006,petersen,LarsenIm,ImLarsen12}
to mention a few examples.

In all cases known to us where one actually succeeded in proving
that over a field every abelian variety has infinite rank,
these fields turn out to be ample or are conjectured to be ample,
where a field $F$ is called {\bf ample} (or {\bf large} following \cite{Pop})
if every smooth $F$-curve $C$
satisfies $C(F)=\emptyset$ or $|C(F)|=\infty$.
For example, every separably closed, real closed, or Henselian valued field is ample.
See \cite{Pop,Jardenample,PopHenselian,FehmParan} for many more examples of ample fields,
and 
\cite{BarySorokerFehm,Pop_survey}
for surveys on the importance and ubiquity of ample fields
in contemporary Galois theory and other areas.

In this work, we prove (see Section \ref{sec_main}) our conjecture from \cite{FP}:

\begin{Theorem}\label{Main}
Let $F$ be an ample field that is not algebraic over a finite field
and let $A/F$ be a non-zero abelian variety. Then 
the rank of $A(F)$ is infinite.
\end{Theorem}

In \cite{FP} we proved several special cases of Theorem \ref{Main}.
In particular, we prove the case when $F$ has characteristic zero.
In this work, we complete the case of positive characteristic.
In the proof of Theorem \ref{Main}, 
we use and strengthen results of Ghioca and Moosa \cite{GM} 
on the Mordell-Lang
conjecture in positive characteristic:
The 1-dimensional case of the Mordell-Lang conjecture is reduced to a special
case, which is then solved by recent work of R\"ossler \cite{roessler}.
What we get (see Section \ref{sec_ml}) is the following:

\begin{Theorem}\label{GMthm}
Let $K$ be an algebraically closed field of positive characteristic $p$,
$A/K$ an abelian variety,
$C$ a closed subcurve of $A$,
and $\Lambda\subseteq A(K)$ a subgroup of finite rank. 
If $\Lambda\cap C(K)$ is infinite, 
then $C$ is elliptic or birational to a curve defined over 
the algebraic closure of $\mathbb{F}_p$.
\end{Theorem}

Theorem \ref{Main} reproves and generalizes several
previous infinite rank results.
In Section \ref{sec:app} we give 
a few concrete applications of Theorem \ref{Main}.

\section*{Notation}

\noindent
Let $K$ be a field.
We denote by $\tilde{K}$ the
algebraic closure of $K$,
by $K_{\rm sep}$ the maximal separable,
and by $K_{\rm ins}$ the maximal purely inseparable
extension of $K$ in $\tilde{K}$. 
If $V$ is a $K$-scheme and $F/K$ is a field extension, then
$V_F:=V\times_{{\rm Spec}(K)}{\rm Spec}(F)$. 
If $V$ and $W$ are $K$-schemes and 
$f\colon V\to W$ is a $K$-morphism, then $f_F\colon V_F\to W_F$ stands for the base 
extension of $f$. 
Throughout this paper, a $K$-variety is a 
separated algebraic $K$-scheme which is geometrically integral, and 
a $K$-curve is a $K$-variety of dimension $1$. 
The absolute genus of a $K$-curve $C$ is the genus $g(C')$ of a smooth projective model $C'$ of $C_{\tilde{K}}$.
We denote by $J_{C'}$ the Jacobian variety of $C'$.
Subvarieties are always
understood to be closed in the ambient variety. 
If $f\colon V\to W$ is a morphism and $X\subseteq W$, then we define
$f^{-1}(X):=(V\times_W X)_{\rm red}$, the reduced induced scheme structure on the fiber product. 
If $f: G\to H$ is a homomorphism of group schemes over $K$, then 
$\mathrm{ker}(f):=G\times_{H, e} \mathrm{Spec}(K)$ is the scheme theoretic fibre of the unit section $e$.

\section{Isotrivial and special curves}

\noindent
In this section we prove some preliminary results about 
isotrivial curves and special curves
in the sense of Hrushovski, cf.~\cite{Hrushovski}.
The main conclusion here, probably known to experts, 
is that a subcurve of an abelian variety is special
if and only if it is elliptic or isotrivial.

{\em Throughout this section,
let $K/K_0$ be an extension of algebraically closed fields.}

A $K$-variety $V$ is {\bf $K/K_0$-isotrivial} (or simply isotrivial) if
there exists a $K_0$-variety $V_0$ and a $K$-birational map
$V\dashrightarrow V_{0,K}$.
Note that if $V$ is a smooth projective curve, then
$V$ is $K/K_0$-isotrivial if and only if 
there exists a $K_0$-variety $V_0$ and a {\em $K$-isomorphism}
$V\rightarrow V_{0,K}$,
i.e.~$V$ descends to $K_0$.
If $F$ is a subfield of $K$, then we say that an $F$-variety $V$ is $K/K_0$-isotrivial if $V_{K}$ is $K/K_0$-isotrivial.

The following lemma also appears 
as 
\cite[Fact 2.15, Remark 2.17]{BenoistBourscarenPillay},
but we include a full proof for the convenience of the reader:

\begin{Lemma}\label{isogeny}
Let $A/K$ be an abelian variety. Assume that 
there exists an abelian variety $B_0/K_0$ and an isogeny $\varphi: B_{0,K}\to A$.
If $\dim(A)=1$ or if $\varphi$ is separable, then there exists an abelian variety 
$A_0/K_0$ with $A_{0,K}\cong A$.
\end{Lemma}

\begin{proof} \par\medskip
 The kernel $N:={\rm ker}(\varphi)$ is a finite subgroup
 scheme of $B_{0,K}$. If $\dim(A)=1$, then $\dim(B_0)=1$ and thus there exists a
 finite subgroup scheme $N_0/K_0$ of $B_0$ such that $N=N_{0,K}$,
 see \cite[Theorem 2.3]{conrad2006}. If $\varphi$ is
 separable, then $N/K$ is a finite {\em \'etale} subgroup scheme of $B_{0,K}$ 
 and thus there exists a
 finite subgroup scheme $N_0/K_0$ of $B_0$ such that $N=N_{0,K}$,
 see \cite[Lemma 3.11]{conrad2006}. In both cases it follows that 
$A_0:=B_0/N_0$ is an abelian variety over $K_0$
(see \cite[3.7]{conrad2006} for the existence of the quotient)
and 
$A_{0,K}=B_{0,K}/N_{0,K}=B_{0,K}/N\cong A$, as desired. 
\end{proof}

\begin{Lemma}\label{piisogeny}
Let $A/K$ be an abelian variety. Assume that 
there exist an abelian variety $B_0/K_0$ and a surjective homomorphism
$B_{0,K}\to A$.
Then there exists an abelian variety 
$A_0/K_0$ and a purely inseparable isogeny  $A\to A_{0,K}$.
\end{Lemma}

\begin{proof}
By the Poincare reducibility theorem \cite[12.1]{milne1986b} there exists an abelian subvariety $B'$ of $B_{0,K}$ such that $A$ is $K$-isogenous to $B'$. 
By \cite[3.21]{conrad2006} there exits an abelian subvariety $B'_0$ of $B_0$ such that $B'\cong B'_{0,K}$. After replacing $B_0$ by $B_0'$ we can assume that there exists an isogeny $\varphi: B_{0,K}\to A$. We put $B:=B_{0,K}$. 

There exists $n\in \Nn$ such that $\ker(\varphi)$ is a subgroup scheme of $\mathrm{ker}([n]_B)$. We define $N_0:=\mathrm{ker}([n]_{B_0})$ and $N:=N_{0,K}$. The projection $p: B\to B/N$ 
induces an isogeny $f: A\to B/N$ such that $f\circ \varphi=p$ (see \cite[3.8]{conrad2006}). 
We denote by 
${\rm ker}(f)^\circ$ the connected component of the unit element of ${\rm ker}(f)$
and by
$f_i: A\to A':=A/\ker(f)^\circ$ the canonical projection. 
Then $f$ factors through $A'$, i.e. there exists an isogeny $f_s: A'\to B/N$ such 
that $f_s\circ f_i=f$ (see \cite[3.7]{conrad2006}). We now have constructed isogenies 
$$
 B\buildrel \varphi\over \longrightarrow A
\buildrel f_i \over \longrightarrow A' 
\buildrel f_s \over \longrightarrow B/N
$$
such that $f_s\circ f_i\circ \varphi = p$. The restriction $f_i\circ \varphi |N^\circ \to \ker(f_s)$ is
trivial because $N^\circ$ is connected and $\ker(f_s)\cong \ker(f)/\ker(f)^\circ$ is \'etale (see \cite[3.10]{conrad2006}). It follows that $\ker(f_i\circ \varphi)\supset N^\circ$ and thus 
$f_i\circ \varphi$ factors through an isogeny $g: B/N^\circ\to A'$ by \cite[3.7]{conrad2006}. The kernel of 
$g$ is a closed subgroup scheme of the finite \'etale group scheme $N/N^\circ$, hence $\ker(g)$
is \'etale and $g: B/N^\circ \to A'$ is a {\em separable} isogeny. But $B/N^\circ= (B_0/(N_0)^\circ)_K$ and thus Lemma \ref{isogeny}
implies that there exists an abelian variety $A_0/K_0$ such that $A'\cong A_{0,K}$. 
\end{proof}

\begin{Lemma}\label{isotrivial}
Let $\varphi\colon D\dashrightarrow C$ be a dominant rational map of $K$-curves.
If $D$ is $K/K_0$-isotrivial, then $C$ is $K/K_0$-isotrivial.
\end{Lemma}

\begin{proof}
Since $K$ and $K_0$ are algebraically closed it suffices to consider the case that $K/K_0$ is transcendental. 
Without loss of generality we can assume that $C$ and $D$ are smooth projective
and that $\varphi:D\rightarrow C$ is a non-constant morphism.
Let $D_0$ be a smooth projective $K_0$-curve with 
an isomorphism $\psi:D_{0,K}\rightarrow D$.

If $g(C)=0$, then $C$ is isomorphic to $\Pp^1_{K}$, hence $C$ is $K/K_0$-isotrivial. 

If $g(C)=1$, then $C$ is an elliptic curve. By Albanese functoriality there is
a surjective homomorphism of abelian varieties 
$(J_{{D}_0})_K\cong J_{D}\to C$, 
and therefore
Lemma \ref{piisogeny} implies that there exists an elliptic curve $C_0/K_0$ such that $C$ is $K$-isogenous to
$C_{0,K}$.
Lemma \ref{isogeny} then shows that $C$ is $K/K_0$-isotrivial.

We may thus assume that $g(C)\ge 2$. 
Let $K_0\subseteq F\subseteq K$ be an intermediate field 
such that
$F/K_0$ is finitely generated and transcendental,
and $C$, $D$, $\varphi$ and $\psi$ descend to $F$,
i.e.~there exist smooth projective $F$-curves $C_1$ and $D_1$
with $C_{1,K}=C$ and $D_{1,K}=D$,
a morphism $\varphi_1:D_1\rightarrow C_1$ such that
$\varphi_{1,K}$ corresponds to $\varphi$, 
and an isomorphism $D_{0,F}\cong D_1$.
Since $K_0$ is algebraically closed,
$|D_1(F)|=|{D}_0(F)|\geq|{D}_0(K_0)|=\infty$.
Since $\varphi_1: {D}_1(F)\rightarrow C_1(F)$ has finite fibres,
this implies that $|C_1(F)|=\infty$.
Thus, since 
$C_1$ is of genus at least $2$,
the 
Grauert-Manin theorem (see \cite{samuelbour} or \cite[Prop. 11.7.1.]{friedjarden})
implies that
$C_1$ is $\tilde{F}/K_0$-isotrivial,
and hence $C$ is $K/K_0$-isotrivial.
\end{proof}

\begin{Lemma} \label{isodown} Let $K_1$ be an algebraically closed intermediate
field of $K/K_0$. Let $C$ be a $K_1$-curve. If $C$ is $K/K_0$-isotrivial, then $C$ is $K_1/K_0$-isotrivial. 
\end{Lemma}

\begin{proof} We can assume that $C$
is smooth and projective. As $C$ is $K/K_0$-isotrivial, there exists a 
$K_0$-curve $C_0$ and a $K$-isomorphism $C_{K}\to C_{0,K}$.
Therefore there exists a finitely generated extension $F/K_1$ inside $K$ and an $F$-isomorphism $f: C_F\to C_{0,F}$. 
Choose a $K_1$-variety $S$ with function field $F$. By \cite[8.8.2, 8.10.5]{EGAIV3} there exists a non-empty open subscheme $S'$ of $S$ such that $f$ extends to an $S$-isomorphism 
$C\times_{K_1} S'\to C_{0, K_1} \times_{K_1} S'$. As $K_1$ is algebraically closed there exists $s\in S'(K_1)$, and specializing at $s$ gives a $K_1$-isomorphism $C\to C_{0,K_1}$. Hence $C$ is $K_1/K_0$-isotrivial.
\end{proof}

Let $A/K$ be an abelian variety. A subvariety $X\subseteq A$ is 
{\bf $K/K_0$-special} 
if there exist 
\begin{enumerate}
 \item[\textbullet] an abelian subvariety $A'$ of $A$, 
  \item[\textbullet] an abelian variety $B_0/K_0$, 
  \item[\textbullet] a $K_0$-subvariety $X_0\subseteq B_0$, 
  \item[\textbullet] a surjective $K$-homomorphism $h\colon A'\to B_{0,K}$,
  \item[\textbullet]  and an element $g\in A(K)$ 
\end{enumerate}
such that $X=g+h^{-1}(X_{0, K})$. 
We write $t_g$ for the translation by $g$ on $A$.

The situation is as follows:
\begin{diagram}[height=0.8cm,width=1.6cm]
A & \lInto & A' & \rTo^h & B_{0,K} \\
\uInto && \uInto && \uInto \\
X & \lTo^{t_g} & h^{-1}(X_{0,K}) & \rTo^h & X_{0,K}\\
\end{diagram}

\begin{Proposition}\label{special}
Let $A/K$ be an abelian variety and $C\subseteq A$ a $K$-subcurve.
Then the following statements are equivalent.
\begin{enumerate}
 \item $C$ is a $K/K_0$-special subvariety of $A$.
 \item $C$ is of absolute genus $1$ or $K/K_0$-isotrivial.
\end{enumerate}
\end{Proposition}

\begin{proof}

{\sc Proof of $(1)\Rightarrow(2)$.}

If $C$ is special,
then there exist an abelian subvariety $A'\subseteq A$,
an abelian variety $B_0/K_0$,
a $K_0$-subvariety $C_0\subseteq B_0$,
a surjective $K$-homo\-mor\-phism $h\colon A'\rightarrow B_{0,K}$,
and an element $g\in A(K)$ such that
$C=g+h^{-1}(C_{0, K})$.
Without loss of generality assume that $0_{B_0}\in C_0$, $0_A\in C$,
and $g=0_A$.
Then $C\subseteq A'$, 
so assume without loss of generality that $A=A'$.

If ${\rm dim}(C_0)=0$,
then $C={\rm ker}(h)$ is an elliptic curve,
hence of absolute genus $1$.

If ${\rm dim}(C_0)\geq1$,
then ${\rm dim}(C_0)={\rm dim}(C)=1$,
so $h$ has finite fibers over $C_{0,K}$,
and hence all fibers of $h$ are finite.
Consequently, 
$h$ is an isogeny.
Let $\hat{h}\colon B_{0,K}\rightarrow A$ be an isogeny
with $h\circ\hat{h}=[n]_{B_{0,K}}$ for some $n\in\mathbb{N}$,
and let $C_0'$ be an irreducible component of $[n]_{B_0}^{-1}(C_0)$ such
that $[n]_{B_0}(C_0')=C_0$. Now we have a non-constant morphism
$$
 C_{0,K}'\buildrel \hat{h}\over \longrightarrow C=h^{-1}(C_{0,K}) 
\buildrel h\over \longrightarrow C_{0,K}
$$
and Lemma \ref{isotrivial} implies that $C$ is $K/K_0$-isotrivial.

\vspace{0.2cm}

{\sc Proof of $(2)\Rightarrow(1)$.}

If $C$ is of absolute genus $1$, then 
the normalization $\lambda: \ol{C}\to C$ of $C$ is an elliptic curve. 
By   \cite[Cor. 2.2]{milne1986b} there exists 
$g\in A(K)$
and a homomorphism $\lambda': \ol{C}\to A$ of 
 abelian varieties such that $\lambda=t_g\circ \lambda'$. 
 It follows that $A':=\mathrm{im}(\lambda')$
 is an abelian subvariety of $A$ and $C=g+A'$. If $B_0$ is the zero abelian variety over $K_0$ and 
 $h: A'\to B_{0,K}$ the zero homomorphism, then $A'=h^{-1}(0_{B_{0,K}})$ and thus $C=g+h^{-1}(0_{B_{0,K}})$ is special.

From now on assume that $C$ is isotrivial. We have to prove that $C$ is special.
There exists a smooth projective $K_0$-curve $C_0$ and a finite birational
morphism $f: C_{0,K}\to C$. 
Let $P_0\in C_0(K_0)$ be a point and let $\alpha_0: C_0\to J_{C_0}$ be the 
corresponding Albanese map (sending $P_0$ to $0$). It is a closed immersion
because the genus of $C_0$ must be at least $1$ by \cite[3.8]{milne1986b}. 
By applying a translation on $A$ we can assume without loss of generality
that $f(P_0)=0_A.$
By Albanese functoriality, the morphism $f$
induces a homomorphism $f'\colon J_{C_0, K}\rightarrow A$
with $f'\circ \alpha_{0,K}=f$. 
Thus $C$ lies on the abelian subvariety 
$A':=f'(J_{C_0,K})$, and $C=f'(\alpha_{0,K}(C_{0,K}))$.
Lemma \ref{piisogeny} gives
an abelian variety $B_0/K_0$ and
a purely inseparable isogeny $h\colon A'\rightarrow B_{0,K}$. Now consider the morphisms
$$
 C_{0,K}\buildrel \alpha_{0,K} \over \longrightarrow
J_{C_0, K}\buildrel f'\over \longrightarrow A'\buildrel h\over\longrightarrow B_{0,K}
$$
By
\cite[20.4(b)]{milne1986b} the homomorphism 
$h\circ f'$
is defined over $K_0$, that is, there exists a homomorphism $\beta_0: J_{C_0}\to B_0$ such that 
$h\circ f'=\beta_{0,K}$. Now $C_0':=\beta_0(\alpha_0(C_0))$ is a subvariety of $B_0$ such that 
$C_{0,K}'=h(f'(\alpha_{0,K}(C_{0,K})))$. 
As $h$ is a purely inseparable isogeny, the topological space underlying each fibre is a singleton. Hence $h: A'\to B_{0,K}$ is a homeomorphism on the underlying topological spaces of $A'$ and $B_{0,K}$.
It follows that 
$$
 h^{-1}(C_{0,K}')=(C_{0,K}'\times_{B_{0,K}} A')_{\mathrm{red}}= f'(\alpha_{0,K}(C_{0,K}))=C.
$$
Thus $C$ is special.
\end{proof}

\section{R\"ossler's theorem on inseparable points on curves}

\noindent
In this section we recall and slightly generalize a recent result of R\"ossler about 
inseparable points on curves,
which strengthens the aforementioned Grauert-Manin theorem.

\begin{Theorem} \label{roes} 
Let $F/F_0$ be a finitely generated extension of fields of positive characteristic, and let $C$ be an $F$-curve of absolute genus $g \ge 2$. If 
$C(F_{\rm ins})$ is infinite, then
$C$ is $\tilde{F}/\tilde{F}_0$-isotrivial.
\end{Theorem}

In the following proof we deduce Theorem \ref{roes} from R\"ossler's \cite[Thm. 1.1]{roessler}, where this theorem is proven 
in the special case where $F_0$ is algebraically closed, 
$\mathrm{trdeg}(F/F_0)=1$ and $C$ is smooth and projective.

\begin{proof}
There exists a regular projective curve over $F_{\mathrm{ins}}$ that is 
birational to $C_{F_{\mathrm{ins}}}$. This curve is automatically geometrically regular because $F_{\mathrm{ins}}$ is perfect (cf. \cite[6.7.7]{EGAIV2}), hence smooth over $F_{\mathrm{ins}}$ by \cite[17.5.2]{EGAIV4}. It follows that there exists a finite extension $E/F$ inside $F_{\mathrm{ins}}$ and a smooth projective $E$-curve $\ol{C}$ that is birational to $C_E$. 
Thus 
$\ol{C}(E_{\rm ins})$ is infinite
and it suffice to prove that
$\ol{C}$ is $\tilde{F}/\tilde{F}_0$-isotrivial.
We can thus assume for the rest of the proof that $C$ is smooth and projective. 
 
Let $\mathcal{F}$ be the set of all
algebraically closed intermediate fields $F'$ of $\tilde{F}/\tilde{F}_0$ with $\mathrm{trdeg}(\tilde{F}/F')=1$. 
Then $\tilde{F}_0=\bigcap_{F'\in \mathcal{F}} F'$:
Indeed, every
$x\in \tilde{F}\setminus \tilde{F}_0$ is an element of a transcendence base $(x, x_2,\cdots, x_n)$ of $\tilde{F}/\tilde{F}_0$ and $x$ is not contained in the algebraic closure $F'$ of $F_0(x_2,\cdots, x_n)$.

For every $F'\in\mathcal{F}$,
$C((F'F)_{\rm ins})\supseteq C(F_{\rm ins})$ is infinite,
so \cite[Thm. 1.1]{roessler} gives that that
$C$ is $\tilde{F}/{F}'$-isotrivial.
Therefore, $C$ is in fact $\tilde{F}/\bigcap_{F'\in\mathcal{F}}F'$-isotrivial:
This can be seen by looking at the point $c\in M_g(\tilde{F})$
corresponding to $C$
on the coarse moduli space $M_g$ for smooth proper curves of genus $g$,
see \cite[5.6, 7.14]{mumfordgeo},
or explicitly by \cite[Corollary 3.2.2]{Koizumi}.
\end{proof}

\begin{Remark}\label{Rem:Kim}
We note that this theorem could also be deduced from 
the older \cite{Kim} instead of \cite{roessler}:
Let $F$ be a function field of one variable over a field $F_0$ of positive characteristic $p$.
Consider an affine 
plane curve 
$$
 C=\Spec(F[X, Y]/(f(X, Y)))
$$ 
where $f(X, Y)\in F[X, Y]$ is absolutely irreducible.
We denote by $f^{(n)}(X, Y)$ the polynomial obtained from $f(X, Y)$ by
raising the coefficients of $f(X, Y)$ to the $p^n$-th power and by
$$
 C^{(n)}:=\Spec(F[X, Y]/(f^{(n)}(X, Y)))
$$ 
the corresponding curve. For 
$n\ge 1$ there is the Frobenius injection
$$
 {\rm Fr}\colon C^{(n-1)}(F)\to C^{(n)}(F),\ (x, y)\mapsto (x^p, y^p)
$$
and we define 
$$
 C^{(n)}(F)_{\rm new}:=C^{(n)}(F)\setminus {\rm Fr}(C^{(n-1)}(F)).
$$
Finally, $C^{(0)}(F)_{\rm new}:=C(F)$.
Note that the Frobenius gives a bijection between $C^{(n)}(F)$ and $C(F^{p^{-n}})$,
and hence between $C^{(n)}(F)_{\rm new}$ and $C(F^{p^{-n}})\setminus C(F^{p^{-n+1}})$.
Then \cite[Theorem 1]{Kim} claims
that if the absolute genus of $C$ is at least $2$
and $C^{(n)}(F)_{\rm new}\neq \emptyset$ for infinitely many
$n$, then $C$ is $\tilde{F}/\tilde{\mathbb{F}}_p$-isotrivial.

This, however, is incorrect, as the following example shows:
Let $F_0$ be any non-perfect transcendental ample field (e.g.~$F_0=\mathbb{F}_p(t)_{\rm sep}$),
$F$ any function field of one variable over $F_0$ (e.g.~$F=F_0(X)$),
and $C$ any $F_0$-curve of absolute genus at least $2$ which is not $\tilde{F}_0/\tilde{\mathbb{F}}_p$-isotrivial
and has a smooth $F_0$-rational point
(see e.g.~Lemma \ref{nonisotrivial} below for the existence of such a curve).
Then \cite[Theorem 1]{Kim} would imply that 
there exists an $n$ such that
$C(F_{\rm ins})=C(F^{p^{-n}})$,
hence $C(F_{0, \rm ins})=C(F_{0, \rm ins})\cap C(F^{p^{-n}})=C(F_0^{p^{-n}})$,
contradicting for example \cite[Corollary 5]{subfields}.
\end{Remark}

\section{The Mordell-Lang Conjecture}
\label{sec_ml}

\noindent
The main ingredient in the proof of Theorem \ref{Main}
is a partial result towards the general Mordell-Lang conjecture
for function fields.

\begin{Conjecture}[General Mordell-Lang conjecture for function fields]
Let $K/K_0$ be an extension of algebraically closed fields,
$A/K$ an abelian variety, 
$X\subseteq A$ a subvariety,
and $\Lambda\subseteq A(K)$ a subgroup of finite rank
such that $\Lambda\cap X(K)$ is Zariski-dense in $X$.
Then $X$ is $K/K_0$-special.
\end{Conjecture}

A slightly weaker result was proven by Hrushovski in \cite{Hrushovski},
where the condition that $\Lambda$ has finite rank is replaced
by the stronger condition that $\Lambda\otimes\mathbb{Z}_{(p)}$
is a finitely generated $\mathbb{Z}_{(p)}$-module,
where $p$ is the characteristic of $K$.

In \cite{GM}, Ghioca and Moosa use work of Scanlon and the theory of generic difference fields
to reduce the general Mordell-Lang conjecture
to the special case where $\Lambda\subseteq A(F_{\rm ins})$
for some finitely generated field $F$.
This reduction also works if on both sides one adds the condition
that $X$ is a curve.
The following result is a slight generalization of \cite[Corollary 2.6]{GM}:

\begin{Proposition}\label{GM}
Let $K$ be an algebraically closed field of positive characteristic $p$,
and let $K_0=\tilde{\mathbb{F}}_p$.
Let $A/K$ be an abelian variety,
$C$ a subcurve of $A$,
and $\Lambda\subseteq A(K)$ a subgroup of finite rank. 
If $\Lambda\cap C(K)$ is infinite, 
then the absolute genus of $C$ is $1$ or $C$ is $K/K_0$-isotrivial,
i.e.~$C$ is $K/K_0$-special.
\end{Proposition}

\begin{proof}
First of all, $A$ and  $C$ are defined over a finitely generated 
subfield $F$ of $K$. 
By the reduction \cite[Theorem 2.2]{GM}, we may assume that 
$\Lambda\subseteq A(F_{\rm ins})$. 
This implies that $C(F_{\rm ins})$ is infinite. 

Since $C$ is a subcurve of an abelian variety,
the absolute genus of $C$ is at least $1$, \cite[3.8]{milne1986b}.
If the absolute genus of $C$ is $1$, then $C$ is $K/K_0$-special by Proposition \ref{special}.
If the absolute genus of $C$ is at least 2,
then $C$ is $K/K_0$-isotrivial by Theorem \ref{roes},
and hence $K/K_0$-special by Proposition \ref{special}.
\end{proof}

This proves Theorem \ref{GMthm}.

\begin{Remark}
Proposition \ref{GM} is stated in \cite[Corollary 2.6]{GM}
in the special case $A=J_C$.
The proof there builds on \cite[Theorem 1]{Kim}, cf.~Remark \ref{Rem:Kim}.
\end{Remark}

\section{Proof of Theorem \ref{Main}}
\label{sec_main}

\noindent
The idea of the proof of Theorem \ref{Main} is
to construct a curve on $A$
that is non-special and has a smooth $F$-rational point.
To achieve this, we 
\textquoteleft lift\textquoteright~a non-isotrivial curve from $\mathbb{A}^2$ to $A$.

\begin{Lemma}\label{nonisotrivial}
If $F/F_0$ is a transcendental field extension,
then there exists a smooth $F$-curve $C$ 
of absolute genus at least $2$ with $C(F)\neq\emptyset$ 
such that $C$ is not $\tilde{F}/\tilde{F}_0$-isotrivial.
\end{Lemma}

\begin{proof}
Let $t\in F\setminus\tilde{F}_0$.
First assume that ${\rm char}(F)\neq 2$.
Let 
$$
 E:y^2=f(x)=(x-1)(x+1)(x-t^2)
$$ 
be the elliptic curve over $F$ with $j$-invariant 
$$
 j(E)=\frac{(16t^4 + 48)^3}{64t^8 - 128t^4 + 64}.
$$
Since $j(E)\in F\setminus\tilde{F}_0$,
it follows that $E$ is not $\tilde{F}/\tilde{F}_0$-isotrivial.
The map $(x,y)\mapsto (x^2,y)$
defines a separable morphism from
the curve 
$$
C:y^2=f(x^2)
$$ 
to $E$ which is ramified over $(0,t)\in E(F)$.
Therefore, $C$ is a smooth $F$-curve of absolute genus at least $2$ by the Riemann-Hurwitz formula,
and $C(F)\neq\emptyset$.
By Lemma \ref{isotrivial}, $C$ is not $\tilde{F}/\tilde{F}_0$-isotrivial.

If ${\rm char}(F)=2$, we can proceed similarly and let
$E:y^2+xy=f(x)=x^3+tx+1$ be the elliptic curve with $j$-invariant
$j(E)=(t^2+1)^{-1}$. The map $(x,y)\mapsto(x^3,y)$ defines a separable morphism 
from the curve $C: y^2+x^3y=f(x^3)$ to $E$ that is ramified over $(0,1)\in E(F)$.
\end{proof}

\begin{Lemma}\label{lift}
Let $F$ be a field. 
Let $X$ be an $F$-variety of dimension $n$ and $Y$ an $F$-variety of dimension $m<n$. Let 
$x\in X(F)$ be a smooth point of $X$ and let $y\in Y(F)$ a smooth point of $Y$. Then there 
exists an $F$-subvariety $V \subseteq X$ of dimension $\dim(V)=m$ such that $x$ is a smooth point of $V$,
and a dominant $F$-rational map
$V\dashrightarrow Y$ mapping $x$ to $y$.
\end{Lemma}

\begin{proof}
There exists a smooth open neighbourhood $Y'$ of $y$ in $Y$ and an \'etale morphism $f: Y'\to \Aa^m$, see \cite[17.11.4]{EGAIV4}.
By \cite[I.3.14]{MilneEtale} this morphism $f$ is locally standard \'etale. Hence, after replacing $Y'$ by a smaller open neighbourhood of $y$ if necessary, there exists an affine open subscheme $U=\Spec(R)\subseteq \Aa^m$ such
that $f(Y')\subseteq U$ and $f|Y'\to U$ is standard \'etale, i.e. such that there exist
a polynomial $Q\in R[T]$ and a $U$-isomorphism of $Y'$
with an open subscheme of 
$\Spec(R[T]/(Q))$. 
In particular there exists an immersion

$$
 i: Y'\to \Spec(R[T])=U\times \Aa^1 \to \Aa^{m+1}\to \Aa^n.
$$
There exists a smooth open neighbourhood $X'$ of $x$ in $X$
and an \'etale morphism $g: X'\to \Aa^n$, again by \cite[17.11.4]{EGAIV4}, and we can assume that  $g(x)=i(y)$. We form the fibre product $Z=X'\times_{\Aa^n, g,i} Y'=g^{-1}(Y')$ and consider the diagram
\begin{diagram}
 &Z          &\rTo^{j}   &X'& \subset X\\
 &\dTo_p &                   &\dTo_g\\
 Y\supset & Y'          &\rTo^{i}   &\Aa^n.\\
\end{diagram}
The projection $p$ is \'etale, $Y'$ is smooth and $\dim(Y')=m$. Hence $Z$ is smooth of dimension $m$. Furthermore $j$ is an immersion. Hence we can identify $Z$ with a closed subscheme of an open subscheme $X''$ of $X'$. We have $x\in Z(F)$.  Let $Z_0$ be the connected component of $Z$ passing through $x$ and let $V$ be the closure of $Z_0$ in $X$ (with the reduced induced subscheme structure). Then $Z_0$ is smooth and connected and $Z_0(F)\neq \emptyset$, hence it is geometrically integral,
so also 
$V$ is geometrically integral, see \cite[4.6.3]{EGAIV2}.
Hence $V$ is a {\em subvariety} of $X$ with $\dim(V)=m$. 
Finally the \'etale morphism $p|Z_0\to Y'$ induces a dominant rational map 
$V\dashrightarrow Y$
mapping $x$ to $y$.
\end{proof}

\begin{Theorem}\label{Curve}
Let $F_1$ be a field which is not algebraic over a finite field, 
and let $A/F_1$ be a non-zero abelian variety. 
Then there exists a smooth $F_1$-curve $C_A$ with $C_A(F_1)\neq\emptyset$ and the following property:
If $F$ is an extension of $F_1$ with $|C_A(F)|=\infty$, 
then the rank of $A(F)$ is infinite.
\end{Theorem}

\begin{proof}
Replace $A$ by $A\times A$ to assume without loss of generality
that $A$ is of dimension at least $2$.
Take any smooth $F_1$-curve $C$ 
of absolute genus at least $2$ with $C(F_1)\neq\emptyset$.
By Lemma \ref{lift}
there exist an $F_1$-curve $D\subseteq A$
with a dominant $F_1$-rational map $\varphi\colon D\dashrightarrow C$ 
such that $D$ has a smooth $F_1$-rational point.
By the Riemann-Hurwitz formula,
$D$ is of absolute genus at least $2$.

If $F_1$ is of characteristic zero
and $F$ is an extension of $F_1$ with $|D(F)|=\infty$,
then the Mordell-Lang conjecture in characteristic zero implies that
the rank of $A(F)$ is infinite, cf.~\cite[Theorem 2.6(a)]{FP}.

If $F_1$ is of characteristic $p>0$,
then $K_1=\tilde{F}_1$ is transcendental over $K_0=\tilde{\mathbb{F}}_p$
and so we can choose the curve $C$ to be not $K_1/K_0$-isotrivial,
see Lemma \ref{nonisotrivial},
and then also $D$ is not $K_1/K_0$-isotrivial by Lemma \ref{isotrivial}. 
If $F$ is an extension of $F_1$ with $|D(F)|=\infty$,
let $K=\tilde{F}$ and $\Lambda=A(F)$,
and suppose that the rank of $\Lambda$ is finite.
Since $\Lambda\cap D(K)=D(F)$ is infinite and the absolute genus of $D$ is at least $2$,
Proposition \ref{GM} implies
that $D$ is $K/K_0$-isotrivial,  and so Lemma \ref{isodown} implies that
$D$ is $K_1/K_0$-isotrivial, a contradiction.
Thus the rank of $A(F)$ must be infinite.

Therefore, in both cases
$D$ contains a smooth open subcurve $C_A$ that satisfies
the requirements. 
\end{proof}

Theorem \ref{Main} is an immediate consequence of
Theorem \ref{Curve} for 
$F_1=F$,
since if $F$ is ample, then $|C_A(F)|=\infty$.

\section{Corollaries and Applications}
\label{sec:app}

\subsection{Uncountable rank}

It turns out that in the situation of Theorem \ref{Main},
not only is the rank of $A(F)$ infinite, but it is as large as possible.

\begin{Corollary}
Let $F$ be an ample field that is not algebraic over a finite field
and let $A/F$ be a non-zero abelian variety. Then 
the rank of $A(F)$ equals the cardinality of $F$.
\end{Corollary}

\begin{proof}
If $F$ is countable, then the claim follows from Theorem \ref{Main}.
If $F$ is uncountable, then $|A(F)|=|F|$ by \cite[Prop. 3.3]{Harbater}.
Since $A_{\rm tor}(\tilde{F})$ is countable, this implies that
$A(F)\otimes\mathbb{Q}$ cannot have a basis of less than $|F|$ elements.
\end{proof}

\subsection{Large extensions of Hilbertian fields}

In \cite[Theorem 1.4]{petersen}, the second author generalizes \cite[Theorem 9.1]{FreyJarden1974}
from finitely generated fields to arbitrary Hilbertian fields.
This generalization also follows easily from our results.
In the following, {\bf almost all}
is to be understood in the sense of Haar measure,
and for an $e$-tuple $\bfsigma\in\Gal(F)^e$,
where $\Gal(F)={\rm Aut}(F_{\rm sep}/F)$ is the absolute Galois group of $F$,
$F_{\rm sep}(\bfsigma)$ is the fixed field in $F_{\rm sep}$
of the group generated by $\sigma_1,\dots,\sigma_e$.  

\begin{Corollary}
If $F$ is a Hilbertian field, $A/F$ a non-zero abelian variety,
and $e$ a non-negative integer,
then $A(F_{\rm sep}(\bfsigma))$ has infinite rank
for almost all $\bfsigma\in\Gal(F)^e$.
\end{Corollary}

\begin{proof}
Let $C_A$ be the $F$-curve associated to $A$ by Theorem \ref{Curve}.
Since $F$ is Hilbertian, there exists a linearly disjoint sequence
$F_1,F_2,\dots$ of extensions of $F$ of fixed degree $d$ such that
$C_A(F_i)\setminus C_A(F)\neq\emptyset$ for all $i$, cf.~\cite[Proof of Theorem 18.6.1, Part A]{friedjarden}.
By Borel-Cantelli, the set of $\bfsigma\in\Gal(F)^e$ for which
$F_{\rm sep}(\bfsigma)$ contains infinitely many of the $F_i$
has measure one, cf.~\cite[18.5.3(b)]{friedjarden}.
For those $\bfsigma$ we have that $|C_A(F_{\rm sep}(\bfsigma))|=\infty$,
and hence that the rank of $A(F_{\rm sep}(\bfsigma))$
is infinite by Theorem \ref{Curve}.
\end{proof}

We do not know if almost all $F_{\rm sep}(\bfsigma)$ are ample
if $F$ is not countable, therefore we had to apply Theorem \ref{Curve},
and not Theorem \ref{Main} directly.

\subsection{Fields of totally $S$-adic numbers}

We give a 
simplified proof and a generalization 
of \cite[Theorem 1.3]{FP}.
If $S$ is a finite set of local primes of a field $F$
(see \cite[p.2]{FP}),
the field $F_{{\rm tot},S}$ of {\bf totally $S$-adic numbers} over $F$ is defined as
the maximal Galois extension of $F$ contained in all completions 
$\hat{F}_{\mf p}$, ${\mf p}\in S$.
For  $\bfsigma\in\Gal(F)^e$,
let $F_{\rm sep}[\bfsigma]$ be the maximal Galois extension of $F$ in
the fixed field $F_{\rm sep}(\bfsigma)$.

\begin{Corollary}\label{QtotS}
Let $F$ be a countable Hilbertian field,
$S$ a finite set of local primes of $F$,
and $e$ a non-negative integer.
If $A/F$ is a non-zero abelian variety,
then $A(F_{{\rm tot},S}\cap F_{\rm sep}[\bfsigma])$ has infinite rank
for almost all $\bfsigma\in\Gal(F)^e$.
\end{Corollary}

\begin{proof} 
Geyer and Jarden prove in \cite{GeyerJarden2002} that the field
$F_{{\rm tot},S}\cap F_{\rm sep}[\bfsigma]$ is ample for almost all 
$\bfsigma\in \Gal(F)^e$. Hence, Theorem \ref{Main} implies the claim.
\end{proof}

This also generalizes 
\cite[Theorem 2.1]{FreyJarden1974}, \cite[Theorem B]{geyerjarden2006}, and
a special case of \cite[Theorem A.1]{petersen}.

\subsection{Larsen's conjecture}

Our last application of Theorem \ref{Main} is a conditional result.
Among the numerous conjectures about which fields are ample,
the following conjecture (see e.g.~\cite{JunkerKoenigsmann}, \cite[\S4.2]{BarySorokerFehm})
is particularly interesting for the questions considered in this work.

\begin{Conjecture}[Koenigsmann]\label{Koenigsmann}
If $F$ is an infinite field with finitely generated absolute Galois group,
then $F$ is ample.
\end{Conjecture}

One of the consequences of Theorem \ref{Main}
is that a proof of Conjecture \ref{Koenigsmann} would immediately prove
the following conjecture of Larsen, see
\cite[Question 1]{Larsen2003}, \cite{Im2006}, \cite{LarsenIm},
\cite{ImLarsen19}.

\begin{Conjecture}[Larsen]\label{Larsen}
Let $F$ be an infinite finitely generated field,
$A/F$ a non-zero abelian variety,
and $e$ a non-negative integer.
Then $A(F_{\rm sep}(\bfsigma))$ has infinite rank
for {\em all} $\bfsigma\in{\rm Gal}(F)^e$.
\end{Conjecture}
 
\begin{Corollary}
Conjecture \ref{Koenigsmann} implies Conjecture \ref{Larsen}.
\end{Corollary}

\begin{proof}
The absolute Galois group of $F_{\rm sep}(\bfsigma)$ is generated by $\bfsigma$,
and hence is finitely generated. 
Therefore, Theorem \ref{Main} implies the claim.
\end{proof}

\section*{Acknowledgments}

\noindent
We would like to thank Moshe Jarden for his constant  encouragement regarding this paper.
This work was supported by the European Commission under contract MRTN-CT-2006-035495,
and the Minkowski Center for Geometry at Tel Aviv University, established by the Minerva Foundation. 
S.P. was supported by the NCN grant OPUS16 UMO 2018/31/B/ST1/01474 of the 
Polish ministry of science.

\bigskip

\end{document}